\documentclass[a4paper,11pt]{amsart}
\usepackage{cellspace} 
\usepackage{amssymb}
\usepackage{fixltx2e}
\usepackage{enumerate}
\usepackage[utf8x]{inputenc}
\usepackage[T1]{fontenc}
\usepackage{bbm}
\usepackage{gfsartemisia-euler}
\usepackage{mathtools}
\usepackage{hyperref}
\usepackage[german,english]{babel}
\hypersetup{%
colorlinks=true,%
linkcolor=black,%
anchorcolor=black,%
citecolor=black,%
menucolor=black,%
urlcolor=black,%
baseurl={http://carva.org/emmanuel.royer},%
pdftitle={Spiegelungssatz: a combinatorial proof for the 4-rank},%
pdfauthor={Laurent Habsieger (CNRS, Universit\'e Claude Bernard) Emmanuel Royer (Universit\'e Blaise Pascal, CNRS)},%
pdfsubject={4-rank, Spiegelungssatz, reflection principle},%
pdfkeywords={4-rank, Spiegelungssatz, combinatorial interpretation, reflection principle},%
a4paper,%
unicode=true,%
pdflang=en,
}
\newtheoremstyle{erdfn}
  {}
  {}
  {\itshape}
  {}
  {\fontseries{bx}\selectfont\itshape}
  {--}
  { }
  {}
\newtheoremstyle{erthm}
  {}
  {}
  {\itshape}
  {}
  {\fontseries{bx}\selectfont\itshape}
  {--}
  { }
  {}
\newtheoremstyle{errem}
  {}
  {}
  {}
  {}
  {\ttfamily\itshape}
  {--}
  { }
  {}
\theoremstyle{erthm}
\newtheorem{thm}{Theorem}
\newtheorem{lem}[thm]{Lemma}
\newtheorem{cor}[thm]{Corollary}
\newtheorem{prop}[thm]{Proposition}
\theoremstyle{erdfn}

\theoremstyle{errem}
\newtheorem*{rem}{Remark}
\makeatletter
 \newlength{\h@uteurnumerateur}
 \newlength{\h@uteurdenominateur}
\newcommand*{\quotientdroite}[2]{\mathop{
  \mathchoice%
  {
    \settoheight{\h@uteurnumerateur}{\ensuremath{\displaystyle{#1#2}}}%
    \settoheight{\h@uteurdenominateur}{\ensuremath{\displaystyle{#1#2}}}%
    \raisebox{0.5\h@uteurnumerateur}{\ensuremath{\displaystyle{#1}}}%
    \mkern-5mu\diagup\mkern-4mu%
    \raisebox{-0.5\h@uteurdenominateur}{\ensuremath{\displaystyle{#2}}}%
  }
  {
    \settoheight{\h@uteurnumerateur}{\ensuremath{\textstyle{#1#2}}}%
    \settoheight{\h@uteurdenominateur}{\ensuremath{\textstyle{#1#2}}}%
  \raisebox{0.2\h@uteurnumerateur}{\ensuremath{\textstyle{#1}}}%
  /%
  \raisebox{-0.2\h@uteurdenominateur}{\ensuremath{\textstyle{#2}}}%
  }
  {
    \settoheight{\h@uteurnumerateur}{\ensuremath{\scriptstyle{#1#2}}}%
    \settoheight{\h@uteurdenominateur}{\ensuremath{\scriptstyle{#1#2}}}%
  \raisebox{0.2\h@uteurnumerateur}{\ensuremath{\scriptstyle{#1}}}%
  /%
  \raisebox{-0.2\h@uteurdenominateur}{\ensuremath{\scriptstyle{#2}}}%
  }
  {
    \settoheight{\h@uteurnumerateur}{\ensuremath{\scriptscriptstyle{#1#2}}}%
    \settoheight{\h@uteurdenominateur}{\ensuremath{\scriptscriptstyle{#1#2}}}%
  \raisebox{0.2\h@uteurnumerateur}{\ensuremath{\scriptscriptstyle{#1}}}%
  /%
  \raisebox{-0.2\h@uteurdenominateur}{\ensuremath{\scriptscriptstyle{#2}}}%
  }
}}
\newcommand*{\quotientgauche}[2]{
  \mathchoice%
  {
    \settoheight{\h@uteurnumerateur}{\ensuremath{\displaystyle{#1#2}}}%
    \settoheight{\h@uteurdenominateur}{\ensuremath{\displaystyle{#1#2}}}%
    \raisebox{-0.5\h@uteurnumerateur}{\ensuremath{\displaystyle{#1}}}%
    \mkern-3mu\diagdown\mkern-5mu%
    \raisebox{0.5\h@uteurdenominateur}{\ensuremath{\displaystyle{#2}}}%
  }
  {
    \settoheight{\h@uteurnumerateur}{\ensuremath{\textstyle{#1#2}}}%
    \settoheight{\h@uteurdenominateur}{\ensuremath{\textstyle{#1#2}}}%
  \raisebox{-0.2\h@uteurnumerateur}{\ensuremath{\textstyle{#1}}}%
  \backslash%
  \raisebox{0.2\h@uteurdenominateur}{\ensuremath{\textstyle{#2}}}%
  }
  {
    \settoheight{\h@uteurnumerateur}{\ensuremath{\scriptstyle{#1#2}}}%
    \settoheight{\h@uteurdenominateur}{\ensuremath{\scriptstyle{#1#2}}}%
  \raisebox{-0.2\h@uteurnumerateur}{\ensuremath{\scriptstyle{#1}}}%
  \backslash%
  \raisebox{0.2\h@uteurdenominateur}{\ensuremath{\scriptstyle{#2}}}%
  }
  {
    \settoheight{\h@uteurnumerateur}{\ensuremath{\scriptscriptstyle{#1#2}}}%
    \settoheight{\h@uteurdenominateur}{\ensuremath{\scriptscriptstyle{#1#2}}}%
  \raisebox{-0.2\h@uteurnumerateur}{\ensuremath{\scriptscriptstyle{#1}}}%
  \backslash%
  \raisebox{0.2\h@uteurdenominateur}{\ensuremath{\scriptscriptstyle{#2}}}%
  }
}
\makeatother
\DeclarePairedDelimiter\abs{\lvert}{\rvert}
\newcommand*{\Af}{\mathcal{F}_D}
\newcommand*{\be}{\beta}
\newcommand*{\ClK}{\mathcal{C\ell}_{\K}}
\newcommand*{\idK}{\mathcal{I}_{\K}}
\newcommand*{\E}{\mathcal{E}_D}
\newcommand*{\F}{\mathbbmss{F}}
\newcommand*{\J}[2]{\left(\frac{\vphantom{X}#1}{\vphantom{X}#2}\right)}
\newcommand*{\K}{\mathbbmss{K}}
\newcommand*{\N}{\mathbbmss{N}}
\newcommand*{\prK}{\mathcal{P}_{\K}}
\newcommand*{\Q}{\mathbbmss{Q}}
\newcommand*{\reK}{\mathbbmss{K}^{\#}}
\DeclareMathOperator{\rk}{Rank}
\newcommand*{\rkq}{\rk_4}
\title{Spiegelungssatz: a combinatorial proof for the \(4\)-rank}
\thanks{This research was partially supported by ANR grant \emph{Modunombres}.}
\thanks{We would like to thank \'Etienne Fouvry for having introduced us to this problem.}
\author[Laurent Habsieger]{Laurent Habsieger}
\address{Universit\'e de Lyon \\
CNRS \\
Universit\'e Lyon 1 \\
INSA  \\
Ecole Centrale de Lyon \\
UMR5208, Institut Camille Jordan \\
43 blvd du 11 novembre 1918 \\
F-69622 Villeurbanne-Cedex, France}
\email{laurent.habsieger@math.univ-lyon1.fr}

\author[Emmanuel Royer]{Emmanuel Royer}
\address{%
Emmanuel Royer\\
Clermont Universit\'e\\
Universit\'e Blaise Pascal\\
Laboratoire de math\'ematiques\\
BP 10448\\
F-63000 Clermont-Ferrand\\
France %
}
\curraddr{%
Emmanuel Royer\\
Universit\'e Blaise Pascal\\
Laboratoire de math\'ematiques\\
Les C\'ezeaux\\
BP 80026\\
F-63171 Aubi\`ere Cedex\\
France %
}
\email{{emmanuel.royer@math.univ-bpclermont.fr}}
\begin{document}
\mathtoolsset{showonlyrefs,mathic}
\begin{abstract}
The  \foreignlanguage{german}{Spiegelungssatz} is an inequality between the \(4\)-ranks of the narrow ideal class groups of the quadratic fields \(\Q(\sqrt{D})\) and \(\Q(\sqrt{-D})\). We provide a combinatorial proof of this inequality. Our interpretation gives an affine system of equations that allows to describe precisely some equality cases. %
\end{abstract}
\maketitle
\tableofcontents%
\section*{Introduction} %
Let \(\K\) be a quadratic field. Let \(\idK\) be the multiplicative group of fractional nonzero ideals of the ring of integers of \(\K\) and \(\prK\) be the subgroup of principal fractional ideals. We consider the subgroup \(\prK^+\) of \(\prK\), whose elements are the ones generated by an element with positive norm. The narrow class group \(\ClK^+\) of \(\K\) is the quotient \(\quotientdroite{\idK}{\prK^+}\). If \(\K\) is imaginary, this is the usual class group \(\ClK\coloneqq\quotientdroite{\idK}{\prK}\) whereas if \(\K\) is real, the group \(\ClK\) is a quotient of \(\ClK^+\). We have \(\ClK^+=\ClK\) if and only if the fundamental unit of \(\K\) has
norm \(-1\). Otherwise, the cardinalities of these two groups differ by a factor \(2\).  For more details about the relations between \(\ClK\) and \(\ClK^+\) we refer to~\cite[Section 3.1]{MR2726105}. The narrow class-group being finite, we can define its \(p^k\)-rank for any power of a prime number \(p^k\) by %
\[%
 \rk_{p^k}(\K)\coloneqq\dim_{\F_p}\quotientdroite{\left(\ClK^+\right)^{p^{k-1}}}{\left(\ClK^+\right)^{p^{k}}}. %
\]
In other words, \(\rk_{p^k}(\K)\) is the number of elementary divisors of \(\ClK^+\) divisible by \(p^k\). %

If \(\K=\Q(\sqrt{\Delta})\), the \emph{reflection} of \(\K\) is the quadratic field \(\reK\coloneqq\Q(\sqrt{-\Delta})\). Assume that \(\K\) is totally real, in \cite[Th\'eor\`emes II.9 and II.10]{MR0280466}, Damey \& Payan proved the following inequality (the so called \emph{\foreignlanguage{german}{Spiegelungssatz}} for the \(4\)-rank, see~\cite{MR0096633}): %
\[%
 \rkq(\K)\leq\rkq(\reK)\leq\rkq(\K)+1. %
\]
In this article, we provide a combinatorial proof of this \emph{\foreignlanguage{german}{Spiegelungssatz}} using expressions involving character sums due to Fouvry \& Kl\"uners \cite{MR2276261}. The letter \(D\) will always denote a positive, odd, squarefree integer. %

Let \(d_{\K}\) be the discriminant of the real quadratic field \(\K\) and \(d_{\K}^{\#}\) be the discriminant of the imaginary quadratic field \(\reK\). The usual computation of the discriminant allows to consider three families of quadratic fields. This families are described table~\ref{tab_link}. 

\begin{table}[h]
 \begin{tabular}{|Sc||Sc|Sc|Sc|} %
  \hline %
  \(d_{\K}\) & \(1\pmod{4}\) & \(0\pmod{8}\) & \(4\pmod{8}\)\\ \hline\hline %
  \(d_{\K}\) & \(D\) & \(8D\) & \(4D\)\\ \hline %
  \(d_{\K}^{\#}\) & \(-4D\) & \(-8D\) & \(-D\)\\ \hline %
  \(d_{\K}^{\#}\)    & \(4\pmod{8}\) & \(0\pmod{8}\) & \(1\pmod{4}\)\\ \hline %
  \(D  \)    & \(1\pmod{4}\) & & \(-1\pmod{4}\)\\ \hline %
  \(\K \)     & \(\Q(\sqrt{D})\) & \(\Q(\sqrt{2D})\) & \(\Q(\sqrt{D})\)\\ \hline %
 \end{tabular}
\caption{Link between \(D\), \(d_{\K}\) and their reflections.}
\label{tab_link}
\end{table} 

We introduce for any integers \(u\) and \(v\) coprime with \(D\) the cardinality %
\[%
 \E(u,v)\coloneqq\#\{(a,b)\in\N^2\colon D=ab,\, ua\equiv\square\pmod{b},\, vb\equiv\square\pmod{a}\}
\]
where \(x\equiv\square\pmod{y}\) means that \(x\) is the square of an integer modulo \(y\). Using table~\ref{tab_link}, we find in \cite{MR2276261} (where what the authors note \(D\) is what we note \(d_{\K}\) or \(d_{\K}^{\#}\)) the following expressions for the \(4\)-rank of \(\K\) and \(\reK\). %
\begin{enumerate}[\indent 1)]
 \item If \(d_{\K}\equiv 1\pmod{4}\), then %
\begin{equation*}
 2^{\rkq(\K)}=\frac{1}{2}\E(-1,1) %
\end{equation*}
\cite[Lemma 27]{MR2276261} %
and %
\begin{equation*}
 2^{\rkq(\reK)}=\frac{1}{2}\left(\E(1,1)+\E(2,2)\right) %
\end{equation*}
\cite[Lemma 40]{MR2276261} with \(D\equiv 1\pmod{4}\). %
 \item If \(d_{\K}\equiv 0\pmod{8}\), then %
\begin{equation*}
 2^{\rkq(\K)}=\frac{1}{2}\left(\E(-2,1)+\E(-1,2)\right) %
\end{equation*}
\cite[Lemma 38]{MR2276261} %
and %
\begin{equation*}
 2^{\rkq(\reK)}=\E(2,1) %
\end{equation*}
\cite[Lemma 33]{MR2276261}. %
 \item If \(d_{\K}\equiv 4\pmod{8}\), then %
\begin{equation*}
 2^{\rkq(\K)}=\frac{1}{2}\left(\E(-1,1)+\E(-2,2)\right) %
\end{equation*}
\cite[Lemma 42]{MR2276261} %
and %
\begin{equation*}
 2^{\rkq(\reK)}=\frac{1}{2}\E(1,1) %
\end{equation*}
\cite[Lemma 16]{MR2276261} with \(D\equiv 3\pmod{4}\). 
\end{enumerate}
\begin{rem}%
These expressions of \(2^{\rkq(\K)}\) and \(2^{\rkq(\reK)}\) either have one term or are a sum of two terms. In case they have one term, it can not be zero and this term is a power of \(2\). In case they are sum of two terms, we will show that each of these terms is either zero or a power of two ; then considering the solutions of the equation \(2^a=2^b+2^c\), we see that either one term (and only one) is zero or the two terms are equal. %
\end{rem}
To prove Damey \& Payan \emph{\foreignlanguage{german}{Spiegelungssatz}}, we have then to prove the three following inequalities. %
\begin{enumerate}[\indent 1)]%
 \item If \(D\equiv 1\pmod{4}\) then %
\begin{equation}\label{eq_DPun}%
 \E(-1,1)\leq\E(1,1)+\E(2,2)\leq2\E(-1,1). %
\end{equation}
 \item For any \(D\),  %
\begin{equation}\label{eq_DPdeux}%
 \E(-2,1)+\E(-1,2)\leq 2\E(2,1)\leq2\E(-2,1)+2\E(-1,2). %
\end{equation}
\item If \(D\equiv 3\pmod{4}\) then %
\begin{equation}\label{eq_DPtrois}%
 \E(-1,1)+\E(-2,2)\leq \E(1,1)\leq2\E(-1,1)+2\E(-2,2). %
\end{equation}
\end{enumerate}

In section~\ref{sec_charsum}, we establish a formula for \(\E(u,v)\) involving Jacobi characters. We average this formula over a group of order \(8\) generated by three permutations. We deduce properties for \(\E(u,v)\) from this formula. In section~\ref{sec_affine}, we give an interpretation of \(\E(u,v)\) in terms of the cardinality of an affine space. In particular, this shows that \(\E(u,v)\) is either \(0\) or a power of \(2\). Finally, in section~\ref{sec_spiegel}, we combine the character sum interpretation with the affine interpretation to deduce the \emph{\foreignlanguage{german}{Spiegelungssatz}}. We also prove the equality cases found by Uehara~\cite[Theorem 2]{MR987569} and give a new one.%
\section{A character sum}\label{sec_charsum}%
Denote by \(\J{m}{n}\) the Jacobi symbol of \(m\) and \(n\), for any coprime odd integers \(m\) and \(n\). The letter \(p\) will always denote a prime number. For any integers \(s\), \(t\), \(u\) and \(v\) coprime with \(D\), we introduce the sum %
\[%
 \sigma_D(s,t,u,v)=\sum_{ab=D}\prod_{p\mid b}\left(\J{s}{p}+\J{ua}{p}\right)\prod_{p\mid a}\left(\J{t}{p}+\J{vb}{p}\right). %
\]
We have %
\[%
 \sigma_D(1,1,u,v)=\sum_{ab=D}\prod_{p\mid b}\left(1+\J{ua}{p}\right)\prod_{p\mid a}\left(1+\J{vb}{p}\right)\eqqcolon S_D(u,v). %
\]
This last sum is nonnegative and related to our problem by the easy equality %
\begin{equation}\label{eq_lienES}%
 \E(u,v)=2^{-\omega(D)}S_D(u,v) %
\end{equation}
where \(\omega(D)\) stands for the number of prime divisors of \(D\). The aim of this section is to establish some properties of \(\sigma_D\). %

We note the symmetry relation %
\begin{equation}\label{eq_sym}%
 \sigma_D(s,t,u,v)=\sigma_D(t,s,v,u) %
\end{equation}
which gives \(S_D(u,v)=S_D(v,u)\). The factorisation %
\begin{equation}\label{eq_facto}%
 \sigma_D(s,t,u,v)=\sum_{ab=D}\J{s}{b}\J{t}{a}\prod_{p\mid b}\left(1+\J{sua}{p}\right)\prod_{p\mid a}\left(1+\J{tvb}{p}\right) %
\end{equation}
implies the upper bound %
\begin{equation}\label{eq_huitL}
 \abs*{\sigma_D(s,t,u,v)}\leq S_D(su,tv). %
\end{equation}
Finally, we shall use the elementary formula %
\begin{equation}\label{eq_fmag}%
 2(-1)^{xy+yz+zx}=(-1)^x+(-1)^y+(-1)^z-(-1)^{x+y+z} %
\end{equation}
valid for any integers \(x,y\) and \(z\). %

We introduce the element \(\be(n)\in\F_2\) by %
\[%
 \J{-1}{n}=(-1)^{\be(n)}.
\]
If \(m\) and \(n\) are coprime, the multiplicativity of the Jacobi symbol gives \(\be(m)+\be(n)=\be(mn)\). With this notation the quadratic reciprocity law reads %
\begin{equation}\label{eq_recqua}%
 \J{m}{n}\J{n}{m}=(-1)^{\be(m)\be(n)}. %
\end{equation}
We shall combine~\eqref{eq_fmag} and~\eqref{eq_recqua} to get the linearisation formula %
\[%
 2\J{x}{y}\J{y}{z}\J{z}{x}\J{x}{z}\J{z}{y}\J{y}{x}=\J{-1}{x}+\J{-1}{y}+\J{-1}{z}-\J{-1}{xyz}. %
\]

\begin{lem}\label{lem_septL}
 For any integers \(s,t,u,v\) coprime with \(D\), the following equality %
\[%
 \sigma_D(s,t,u,v)=\sum_{abcd=D}(-1)^{\be(c)\be(d)}\J{a}{d}\J{b}{c}\J{s}{b}\J{t}{a}\J{u}{d}\J{v}{c} %
\]
holds. %
\end{lem}
\begin{proof}%
 By bimultiplicativity of the Jacobi symbol, equation~\eqref{eq_facto} gives %
\[%
 \sigma_D(s,t,u,v)=\sum_{ab=D}\J{s}{b}\J{t}{a}\sum_{d\mid b}\J{usa}{d}\sum_{c\mid a}\J{tvb}{c}. %
\]
By the change of variables \((a,b,c,d)=(\alpha\gamma,\beta\delta,\gamma,\delta)\), we get %
\[%
 \sigma_D(s,t,u,v)=\sum_{D=\alpha\beta\gamma\delta}\J{s}{\beta}\J{u}{\delta}\J{v}{\gamma}\J{t}{\alpha}\J{\gamma}{\delta}\J{\delta}{\gamma}\J{\alpha}{\delta}\J{\beta}{\gamma}
\]
and we conclude using the quadratic reciprocity law~\eqref{eq_recqua} to \(\J{\gamma}{\delta}\J{\delta}{\gamma}\). %
\end{proof}
To build symmetry, we average the formula in lemma~\ref{lem_septL} over an order \(8\) group, namely the group generated by three permutations: the permutation \((a,d)\), the permutation \((b,c)\) and the permutation \(\left((a,b),(d,c)\right)\). The quadratic reciprocity law allows to factorise the term \((-1)^{\be(c)\be(d)}\J{a}{d}\J{b}{c}\) in every transformed sum and then to see \(u\) and \(v\) as describing the action of each permutation. %
\begin{prop}\label{prop_Spermu}%
For any integers \(s,t,u,v\) coprime with \(D\), the following equality %
\begin{align*}%
8S_D(u,v) &= \sum_{abcd=D}(-1)^{\be(c)\be(d)}\J{a}{d}\J{b}{c}\times \\%
\Bigl[2\J{u}{d}\J{v}{c} %
&+ \J{u}{a}\J{v}{c}\left(\J{-1}{a}+\J{-1}{c}+\J{-1}{d}-\J{-1}{acd}\right)\\
&+ \J{u}{d}\J{v}{b}\left(\J{-1}{b}+\J{-1}{c}+\J{-1}{d}-\J{-1}{bcd}\right)\\
&+ \J{u}{a}\J{v}{b}\left(1+\J{-1}{ac}+\J{-1}{bd}-\J{-1}{D}\right)\Bigr]. %
\end{align*}
holds. %
\end{prop}
\begin{proof}%
 From lemma~\ref{lem_septL} follows %
\begin{equation}\label{eq_deper}%
 S_D(u,v)=\sum_{abcd=D}(-1)^{\be(c)\be(d)}\J{a}{d}\J{b}{c}\J{u}{d}\J{v}{c}. %
\end{equation}
We permute \(a\) and \(d\) and use the quadratic reciprocity law~\eqref{eq_recqua} to obtain %
\begin{multline*}%
 S_D(u,v)=\sum_{abcd=D}(-1)^{\be(c)\be(d)}\J{a}{d}\J{b}{c}\J{u}{a}\J{v}{c}\\\times(-1)^{\be(c)\be(d)+\be(d)\be(a)+\be(a)\be(c)}. %
\end{multline*}
Formula~\eqref{eq_fmag} gives %
\begin{multline}\label{eq_permun}%
 2S_D(u,v)=\sum_{abcd=D}(-1)^{\be(c)\be(d)}\J{a}{d}\J{b}{c}\J{u}{a}\J{v}{c}\\\times%
\left(\J{-1}{a}+\J{-1}{c}+\J{-1}{d}-\J{-1}{acd}\right). %
\end{multline}
Similary, we permute \(b\) and \(c\), then use the quadratic reciprocity law~\eqref{eq_recqua} and formula~\eqref{eq_fmag} to get 
\begin{multline}\label{eq_permdeux}%
 2S_D(u,v)=\sum_{abcd=D}(-1)^{\be(c)\be(d)}\J{a}{d}\J{b}{c}\J{u}{d}\J{v}{b}\\\times%
\left(\J{-1}{b}+\J{-1}{c}+\J{-1}{d}-\J{-1}{bcd}\right). %
\end{multline}
Finally, we permute \((a,b)\) and \((b,c)\), apply twice the quadratic reciprocity law~\eqref{eq_recqua} to get  %
\begin{multline}
 2S_D(u,v)=\sum_{abcd=D}(-1)^{\be(c)\be(d)}\J{a}{d}\J{b}{c}\J{u}{a}\J{v}{b}\\\times%
(-1)^{\be(c)\be(d)+\be(b)\be(a)+\be(a)\be(d)+\be(b)\be(c)}. %
\end{multline}
Since \(\be(c)\be(d)+\be(b)\be(a)+\be(a)\be(d)+\be(b)\be(c)=\be(ac)\be(bd)\), using formula~\eqref{eq_fmag} with \(z=0\) we get
\begin{multline}\label{eq_permtrois}%
 2S_D(u,v)=\sum_{abcd=D}(-1)^{\be(c)\be(d)}\J{a}{d}\J{b}{c}\J{u}{a}\J{v}{b}\\\times%
\left(1+\J{-1}{ac}+\J{-1}{bd}-\J{-1}{D}\right). %
\end{multline}
We obtain the result by adding twice~\eqref{eq_deper} with the sum of~\eqref{eq_permun}, \eqref{eq_permdeux} and~\eqref{eq_permtrois}. %
\end{proof} 
When two expressions are equivalent under the action of the symmetry group, we get an identity. We give two such formulas in the next two corollaries.
\begin{cor}\label{cor_dun}%
 If \(D\equiv 1\pmod{4}\) then \(S_D(-1,1)=S_D(1,1)\). %
\end{cor}
\begin{proof}%
 For any \(D\), we obtain from proposition~\ref{prop_Spermu} the formula %
\begin{multline}\label{eq_difSD}%
 8\left(S_D(1,1)-S_D(-1,1)\right)=\sum_{abcd=D}(-1)^{\be(c)\be(d)}\J{a}{d}\J{b}{c}\times\\ %
\left(1-\J{-1}{D}\right)\left(1+\J{-1}{b}\right)\left(1+\J{-1}{c}\right). %
\end{multline}
This gives the result since \(\J{-1}{D}=1\) if \(D\equiv 1\pmod{4}\). %
\end{proof}
\begin{cor}\label{cor_egFdpm}%
 If \(D\equiv{3}\pmod{4}\) then \(S_D(1,1)=2S_D(-1,1)\). %
\end{cor}
\begin{proof}%
 For any \(D\), proposition~\ref{prop_Spermu} gives %
\begin{multline*}%
 8S_D(1,-1)=\sum_{abcd=D}(-1)^{\be(c)\be(d)}\J{a}{d}\J{b}{c}\Bigl[%
2+\J{-1}{b}+2\J{-1}{c}+\\ \J{-1}{d}+\J{-1}{ac}+\J{-1}{bd}+\J{-1}{bc}-\J{-1}{ad}+\J{-1}{abc}-\J{-1}{acd} %
\Bigr]. %
\end{multline*}
Thanks to~\eqref{eq_difSD}, we deduce for any \(D\) the equality %
\begin{multline*}%
-8\left(S_D(1,1)-S_D(-1,1)-S_D(1,-1)\right) = \sum_{abcd=D}(-1)^{\be(c)\be(d)}\J{a}{d}\J{b}{c}\times \\%
\Biggl[1+\J{-1}{c}+\J{-1}{d}+\J{-1}{ac}+\J{-1}{bd}+\J{-1}{abc}+\J{-1}{abd}+\J{-1}{D}\Biggr]. %
\end{multline*}
It follows that %
\begin{multline*}%
-8\left(S_D(1,1)-S_D(-1,1)-S_D(1,-1)\right) = \sum_{abcd=D}(-1)^{\be(c)\be(d)}\J{a}{d}\J{b}{c}\times \\%
\left(1+\J{-1}{D}\right)\left(1+\J{-1}{c}+\J{-1}{d}+\J{-1}{ac}\right). %
\end{multline*}
This finishes the proof since \(\J{-1}{D}=-1\) if \(D\equiv 3\pmod{4}\). %
\end{proof}
Finally, after having dealt with equalities, we shall need the following inequalities. %
\begin{lem}\label{lem_ineq}%
 For any \(D\), for any \(u\) coprime with \(D\), the following inequalities
\[%
 S_D(u,1)\leq S_D(-u,1)+S_D(u,-1)\leq 2S_D(u,1) %
\]
hold. %
\end{lem}
\begin{proof}%
We prove first the inequality %
\begin{equation}\label{eq_majoam}%
 S_D(-u,1)+S_D(u,-1)\leq 2S_D(u,1). %
\end{equation}
With proposition~\ref{prop_Spermu}, we write %
\begin{multline}%
 8\left(S_D(-u,1)+S_D(u,-1)\right)=\sum_{abcd=D}(-1)^{\be(c)\be(d)}\J{a}{d}\J{b}{c}\times\\ %
\Biggl[2\J{u}{d}\left(1+\J{-1}{c}+\J{-1}{d}+\J{-1}{bd}\right)\\ %
+\J{u}{a}\Biggl(2+\J{-1}{a}+\J{-1}{b}+\J{-1}{c}+\J{-1}{d}+2\J{-1}{ac}\\+\J{-1}{abd}+\J{-1}{abc}-\J{-1}{acd}-\J{-1}{bcd}\Biggl)%
\Biggr]. %
\end{multline}
Using %
\begin{equation}\label{eq_xyzabc}%
\J{-1}{xyz}=\J{-1}{D}\J{-1}{t} %
\end{equation}
for any \(\{x,y,z,t\}=\{a,b,c,d\}\) together with~\eqref{eq_deper} and lemma~\ref{lem_septL} we deduce %
\begin{multline*}%
 8\left(S_D(-u,1)+S_D(u,-1)\right)=2\left(S_D(u,1)+S_D(u,-1)+S_D(-u,1)\right)\\
+2\left(\sigma_D(-1,1,-u,1)+\sigma_D(1,u,1,1)+\sigma_D(1,-u,1,-1)\right)\\%
+\left(1-\J{-1}{D}\right)\left(\sigma_D(1,-u,1,1)+\sigma_D(-1,u,1,1)\right)\\%
+\left(1+\J{-1}{D}\right)\left(\sigma_D(1,u,1,-1)+\sigma_D(1,u,-1,1)\right). %
\end{multline*}
Since \(1-\J{-1}{D}\) and \(1+\J{-1}{D}\) are nonnegative, the upper bound~\eqref{eq_huitL} gives %
\[%
 8\left(S_D(-u,1)+S_D(u,-1)\right)\leq 4\left(2S_D(u,1)+S_D(u,-1)+S_D(-u,1)\right) %
\]
hence~\eqref{eq_majoam}. We prove next the inequality %
\begin{equation}\label{eq_minoam}%
 S_D(u,1)\leq S_D(-u,1)+S_D(u,-1). %
\end{equation}
As for~\eqref{eq_majoam}, we use equation~\eqref{eq_xyzabc}, proposition~\ref{prop_Spermu}, equation~\eqref{eq_deper} and lemma~\ref{lem_septL} to get %
\begin{multline*}
 8S_D(u,1)=2S_D(u,1)+S_D(u,-1)+S_D(-u,1)\\ %
+\sigma_D(1,-u,1,1)+\sigma_D(1,u,1,-1)+\sigma_D(1,u,-1,1)+\sigma_D(-1,1,u,1)\\ %
+\left(1+\J{-1}{D}\right)\sigma_D(1,-u,1,-1)+\left(1-\J{-1}{D}\right)\sigma_D(1,u,1,1)\\%
-\J{-1}{D}\bigl(\sigma_D(-1,u,1,1)+\sigma_D(1,-1,u,1)\bigr). %
\end{multline*}
Then~\eqref{eq_huitL} leads to %
\[%
 8S_D(u,1)\leq4\left(S_D(u,1)+S_D(u,-1)+S_D(-u,1)\right) %
\]
hence~\eqref{eq_minoam}. %
\end{proof}
\section{An affine interpretation}\label{sec_affine}
We write \(p_1<\dotsm<p_{\omega(D)}\) for the prime divisors of \(D\) and define a bijection between the set of divisors \(a\) of \(D\) and the set of sequences \((x_i)_{1\leq i\leq\omega(D)}\) in \(\F_2^{\omega(D)}\) by %
\[%
 x_i=\begin{cases}%
      1 & \text{if \(p_i\mid a\)}\\
      0 & \text{otherwise}.
     \end{cases}
\]
Let \(a\) and \(b\) satisfy \(D=ab\) and \(u\) and \(v\) two integers coprime with \(D\). We extend the notation of the previous section writing %
\[%
 \J{a}{b}=(-1)^{\alpha(a,b)}=(-1)^{\beta_a(b)} %
\]
with \(\alpha(a,b)=\beta_a(b)\in\F_2\). The condition that \(vb\) is a square modulo \(a\) is equivalent to \(\J{vb}{p}=1\) for any prime divisor \(p\) of \(a\), that is %
\[%
\J{v}{p_i}\prod_{j\colon x_j=0}\J{p_j}{p_i}=1 %
\]
for any \(i\) such that \(x_i=1\). With our notation, this gives %
\[%
 \forall i,\, x_i=1\Longrightarrow (-1)^{\beta_v(p_i)}(-1)^{\sum_{j\colon x_j=0}\alpha(p_j,p_i)}=1. %
\]
We rewrite it %
\[%
 \forall i,\, x_i=1\Longrightarrow (-1)^{\beta_v(p_i)}(-1)^{\sum_{j\neq i}(1-x_j)\alpha(p_j,p_i)}=1 %
\]
and so %
\begin{equation}\label{eq_condun}%
  \forall i,\, x_i\beta_v(p_i)+\sum_{j\neq i}x_i(1-x_j)\alpha(p_j,p_i)=0. %
\end{equation}
Similary, the condition that \(ua\) is a square modulo \(b\) is equivalent to %
\begin{equation}\label{eq_conddeux}%
  \forall i,\, (1-x_i)\beta_u(p_i)+\sum_{j\neq i}(1-x_i)x_j\alpha(p_j,p_i)=0. %
\end{equation}
Since \(x_i\) is either \(0\) or \(1\), equations \eqref{eq_condun} et \eqref{eq_conddeux} are equivalent to their sum. We deduce the following lemma. %
\begin{lem}\label{lem_cardi}%
 The cardinality \(\E(u,v)\) is the cardinality of the affine space \(\Af(u,v)\) in \(\F_2^{\omega(D)}\) of equations %
\[%
 \left(\beta_u(p_i)+\beta_v(p_i)+\sum_{j\neq i}\alpha(p_j,p_i)\right)x_i+\sum_{j\neq i}\alpha(p_j,p_i)x_j=\beta_u(p_i) %
\]
for all \(i\in\{1,\dotsc,\omega(D)\}\).
\end{lem}
\begin{rem}
 In particular, lemma~\ref{lem_cardi} shows that \(\E(u,v)\) if not zero is a power of \(2\), the power being the dimension of the direction of \(\Af(u,v)\). This is not \emph{a priori} obvious. %
\end{rem}
\begin{rem}
 This interpretation slightly differs from the one found by Redei~\cite{0009.05101,MR759260}. The matrix with coefficients in \(\F_2\) associated to our affine space is \((a_{ij})_{1\leq i,j\leq\omega(D)}\) with %
\[%
 a_{ij}=\begin{dcases}%
         \alpha(p_j,p_i) & \text{ if \(i\neq j\)}\\
	 \beta_u(p_i)+\beta_v(p_i)+\sum_{\ell\neq i}\alpha(p_\ell,p_i) & \text{ if \(i=j\)}
        \end{dcases}
\]
whereas the matrix considered by Redei is \((\widetilde{a}_{ij})_{1\leq i,j\leq\omega(D)}\) with %
\[%
 \widetilde{a}_{ij}=\begin{dcases}%
         \alpha(p_j,p_i) & \text{ if \(i\neq j\)}\\
	 \omega(D)+1+\sum_{\ell\neq i}\alpha(p_\ell,p_i) & \text{ if \(i=j\).}
        \end{dcases}
\]
\end{rem}

\begin{cor}\label{cor_tard}%
 For any \(D\), we have \(S_D(1,1)\neq 0\) and, either \(S_D(2,2)=0\) or \(S_D(2,2)=S_D(1,1)\). %
\end{cor}
\begin{proof}%
 The affine space \(\Af(2,2)\) has equations %
\[%
 \left(\sum_{j\neq i}\alpha(p_j,p_i)\right)x_i+\sum_{j\neq i}\alpha(p_j,p_i)x_j=\beta_2(p_i) %
\]
for all \(i\in\{1,\dotsc,\omega(D)\}\). The affine space \(\Af(1,1)\) has equations %
\[%
 \left(\sum_{j\neq i}\alpha(p_j,p_i)\right)x_i+\sum_{j\neq i}\alpha(p_j,p_i)x_j=0 %
\]
for all \(i\in\{1,\dotsc,\omega(D)\}\). Hence, both spaces have the same direction, and same dimension. The space \(\Af(1,1)\) is not empty: it contains \((1,\dotsc,1)\). Its cardinality is then \(2^{\dim_{\F_2}\Af(1,1)}\). The affine space \(\Af(2,2)\) might be empty and, if it is not, then its cardinality is \(2^{\dim_{\F_2}\Af(2,2)}=2^{\dim_{\F_2}\Af(1,1)}\). It follows that \(\E(1,1)\neq 0\) and, either \(\E(2,2)=0\) or \(\E(2,2)=\E(1,1)\). We finish the proof thanks to~\eqref{eq_lienES}. %
\end{proof}
\begin{cor}\label{cor_plustard}%
 For any \(D\), we have \(S_D(-1,1)\neq 0\) and, either \(S_D(-2,2)=0\) or \(S_D(-2,2)=S_D(-1,1)\). %
\end{cor}
\begin{proof}%
 Since \(\beta_{-2}(p_i)+\beta_{2}(p_i)=\beta_{-1}(p_i)\), the affine space \(\Af(-2,2)\) has equations %
\[%
 \left(\beta_{-1}(p_i)+\sum_{j\neq i}\alpha(p_j,p_i)\right)x_i+\sum_{j\neq i}\alpha(p_j,p_i)x_j=\beta_{-2}(p_i) %
\]
for all \(i\in\{1,\dotsc,\omega(D)\}\). The affine space \(\Af(-1,1)\) has equations %
\[%
 \left(\beta_{-1}(p_i)+\sum_{j\neq i}\alpha(p_j,p_i)\right)x_i+\sum_{j\neq i}\alpha(p_j,p_i)x_j=\beta_{-1}(p_i) %
\]
for all \(i\in\{1,\dotsc,\omega(D)\}\). Hence, both spaces have the same direction, and same dimension. The space \(\Af(-1,1)\) is not empty: it contains \((1,\dotsc,1)\).  It follows that \(\E(-1,1)\neq 0\) and, either \(\E(-2,2)=0\) or \(\E(-2,2)=\E(-1,1)\). We finish the proof thanks to~\eqref{eq_lienES}. %
\end{proof}

\section{Damey-Payan Spiegelungssatz}\label{sec_spiegel}
\subsection{Proof of the Spiegelungssatz}
We have to prove~\eqref{eq_DPun}, \eqref{eq_DPdeux} and \eqref{eq_DPtrois}. %

Consider the case \(d_{\K}\equiv 1\pmod{4}\). Recall that \(D=d_{\K}\). Thanks to~\eqref{eq_lienES}, equation~\eqref{eq_DPun} is %
\[%
 S_D(-1,1)\leq S_D(1,1)+S_D(2,2)\leq 2S_D(-1,1) %
\]
for any \(D\equiv 1\pmod{4}\). By corollary~\ref{cor_dun}, this inequality is equivalent to \(S_D(2,2)\leq S_D(1,1)\) and this last inequality is implied by corollary~\ref{cor_tard}.

Consider the case \(d_{\K}\equiv 0\pmod{8}\). Recall that \(D=d_{\K}/8\). Thanks to~\eqref{eq_lienES}, equation~\eqref{eq_DPdeux} is %
\[%
 S_D(2,1)\leq S_D(-2,1)+S_D(2,-1)\leq 2S_D(2,1) %
\]
for any \(D\). This is implied by lemma~\ref{lem_ineq} with \(u=2\). %

Finally, consider the case \(d_{\K}\equiv 4\pmod{8}\). Recall that \(D=d_{\K}/4\). Thanks to~\eqref{eq_lienES}, equation~\eqref{eq_DPtrois} is %
\[%
 S_D(-1,1)+S_D(-2,2)\leq S_D(1,1)\leq2S_D(-1,1)+2S_D(-2,2) %
\]
for any \(D\equiv 3\pmod{4}\). By corollary~\ref{cor_egFdpm}, this inequality is equivalent to \(S_D(-2,2)\leq S_D(-1,1)\) and this last inequality is implied by corollary~\ref{cor_plustard}.
\subsection{Some equality cases}
It is clear from our previous computations that %
\begin{itemize}
 \item if \(d_{\K}\equiv 1\pmod{4}\) then %
\[%
 \rkq(\reK)=\begin{cases}%
             \rkq(\K) & \text{if \(\E(2,2)=0\)}\\
	     \rkq(\K)+1 & \text{otherwise;}
            \end{cases}
\]
 \item if \(d_{\K}\equiv 4\pmod{8}\) then %
\[%
 \rkq(\reK)=\begin{cases}%
             \rkq(\K)+1 & \text{if \(\E(-2,2)=0\)}\\
	     \rkq(\K) & \text{otherwise.}
            \end{cases}
\]
\end{itemize}
We do not have such clear criterium in the case \(d_{\K}\equiv 0\pmod{8}\). The reason is that our study of the cases \(d_{\K}\equiv 1\pmod{4}\) and \(d_{\K}\equiv 4\pmod{8}\) rests on equalities (corollaries~\ref{cor_dun}, \ref{cor_egFdpm}, \ref{cor_tard} and \ref{cor_plustard}) whereas, our study of the case \(d_{\K}\equiv 0\pmod{8}\) rests on inequalities (lemma~\ref{lem_ineq} and mainly equation~\eqref{eq_huitL}). We study more explicitely special cases in proving the following proposition due to Uehara~\cite[Theorem 2]{MR987569} (the case~\ref{item_new} seems to be new). %
\begin{thm}%
 Let \(\K\) be a real quadratic field of discriminant \(d_{\K}\) and \(D\) be described in table~\ref{tab_link}. Suppose that every prime divisors of \(D\) is congruent to \(\pm 1\) modulo \(8\). Then %
\begin{enumerate}[\indent a)]%
 \item If \(d_{\K}\equiv 1\pmod{4}\), then \(\rkq(\reK)=\rkq(\K)+1\). %
 \item If \(d_{\K}\equiv 0\pmod{8}\) and \(D\equiv -1\pmod{4}\), then \(\rkq(\reK)=\rkq(\K)+1\). %
 \item\label{item_new} If \(d_{\K}\equiv 0\pmod{8}\) and \(D\equiv 1\pmod{4}\), then \(\rkq(\reK)=\rkq(\K)\). %
 \item If \(d_{\K}\equiv 4\pmod{8}\), then \(\rkq(\K)=\rkq(\reK)\). %
\end{enumerate}
\end{thm}
\begin{proof}%
Since every prime divisors of \(D\) is congruent to \(\pm 1\) modulo \(8\), we have \(\beta_2(p_i)=0\) for any \(i\). %
\begin{itemize}%
\item If \(d_{\K}\equiv 1\pmod{4}\), then \(D\equiv 1\pmod{4}\). By lemma~\ref{lem_cardi}, we know that \(\E(2,2)\) is the cardinality of an affine space having equations %
\[%
 \sum_{j\neq i}\alpha(p_j,p_i)(x_i+x_j)=0\quad(1\leq i\leq\omega(D)) %
\]
hence it is non zero (\(x_i=1\) for any \(i\) gives a solution). %
\item If \(d_{\K}\equiv 0\pmod{8}\), then %
\[%
 2^{\rkq(\reK)-\rkq(\K)}=\frac{2\E(2,1)}{\E(-2,1)+\E(-1,2)}. %
\]
Since \(\beta_{-2}(p_i)=\beta_{-1}(p_i)\) for any \(i\), lemma~\ref{lem_cardi} shows that \(\E(-2,1)=\E(-1,2)=\E(-1,1)\). Lemma~\ref{lem_cardi} also shows that \(\E(2,1)=\E(1,1)\), hence %
\[%
 2^{\rkq(\reK)-\rkq(\K)}=\frac{\E(1,1)}{\E(-1,1)}. %
\]
If \(D\equiv -1\pmod{4}\), corollary~\ref{cor_egFdpm} implies that 
\[%
 2^{\rkq(\reK)-\rkq(\K)}=2 %
\]
whereas, if \(D\equiv 1\pmod{4}\), corollary~\ref{cor_dun} implies that 
\[%
 2^{\rkq(\reK)-\rkq(\K)}=1. %
\]
\item If \(d_{\K}\equiv 4\pmod{8}\), then \(D\equiv -1\pmod{4}\). By lemma~\ref{lem_cardi}, we know that \(\E(-2,2)\) is the cardinality of an affine space having equations %
\[%
 \beta_{-1}(p_i)x_i+\sum_{j\neq i}\alpha(p_j,p_i)(x_i+x_j)=\beta_{-1}(p_i)\quad(1\leq i\leq\omega(D)) %
\]
hence it is non zero (\(x_i=1\) for any \(i\) gives a solution). %
\end{itemize}
\end{proof}
\begin{rem}%
Probabilistic results have been given by Gerth\cite{MR1838376} and, for a more natural probability by Fouvry \& Klüners in \cite{MR2679097}. Among other results, Fouvry \& Klüners prove that %
\begin{multline*}%
 \lim_{X\to+\infty}\frac{\#\left\{d_{\K}\in\mathcal{D}(X)\colon \rkq(\reK)=s \vert \rkq(\K)=r\right\}}{\#\mathcal{D}(X)}\\=\begin{cases}%
                                                                                                                          1-2^{-r-1}&\text{if \(r=s\)}\\
															  2^{-r-1}&\text{if \(r=s-1\)}\\
															  0&\text{otherwise.}
                                                                                                                         \end{cases}
\end{multline*}
where \(\mathcal{D}(X)\) is the set of fundamental discriminants in \(]0,X]\). %
 
\end{rem}

\bibliographystyle{amsalpha}                                                                                                                                 
\bibliography{HaRo2}
\end{document}